\documentclass[a4paper, 12pt]{amsart}
\usepackage{amsmath,latexsym,amssymb}

\newtheorem{lemma}{Lemma}[section]
\newtheorem{theorem}[lemma]{Theorem}
\newtheorem{corollary}[lemma]{Corollary}
\newtheorem{definition}[lemma]{Definition}

\newtheorem{remark}[lemma]{Remark}
\newtheorem{proposition}[lemma]{Proposition}

\def\desc{{\rm desc}}
\def\anc{{\rm anc}}
\def\N{\mathbb{N}}
\def\C{\mathcal{C}}
\def\CB{\bar{\mathcal{C}}}

\def\DG{D_\Gamma}
\def\l{\ell}
\def\Aut{{\rm Aut}}
\def\v{\mathbf{v}}
\def\w{\mathbf{w}}
\def\T{\mathcal{T}}
\def\u{\mathbf{u}}
\def\z{\mathbf{z}}
\usepackage{setspace}

\begin{document}

\title[]{Infinite primitive and distance transitive directed graphs of finite out-valency }

\authors{

\author{Daniela Amato}

\address{%
Departamento de Matematica\\
Universidade de Brasilia\\
Campus Universit\'ario Darcy Ribeiro, Bras\'{\i}lia
CEP 70910-900\\
Brazil. }

\email{d.a.amato@mat.unb.br}

\author{David M. Evans}

\address{%
School of Mathematics\\
University of East~Anglia\\
Norwich NR4~7TJ\\
UK.}

\email{d.evans@uea.ac.uk}

}

\date{}

\begin{abstract} We give certain properties which are satisfied by the descendant set of a vertex in an infinite,  primitive, distance transit\-ive digraph of finite out-valency and provide a strong structure theory for digraphs satisfying these properties. In particular, we show that there are only countably many possibilities for the isomorphism type of such a descendant set, thereby confirming a conjecture of the first Author. As a partial converse, we show that certain related conditions on a countable digraph are sufficient for it to occur as the descendant set of a primitive, distance transitive digraph.
\end{abstract}
\maketitle
\setcounter{footnote}{1}\footnotetext{This work was supported by EPSRC grant EP/G067600/1}
\setcounter{footnote}{2}\footnotetext{Date: 9 August 2012; revised 9 April 2014}

\section{Introduction}

\subsection{Background and main results} We begin with an overview of the paper. Most of the terminology is standard and definitions can be found in the next subsection.  

We are interested in the construction and classification of infinite, vertex transitive directed graphs of finite out-valency whose automorphism groups have additional transitivity properties, such as primitivity, distance transitivity or high arc transitivity. In contrast to the finite case where powerful tools from finite group theory are available, there is no possibility of a complete description of such digraphs. Instead, our results will focus on the structure of the \textit{descendant set} of a vertex in such a digraph: this is the induced subdigraph on the set of vertices reachable from the given vertex by an outward-directed path. Much of the motivation for the work comes from  questions of Peter M. Neumann on infinite permutation groups, and work on highly arc transitive digraphs originating in \cite{Cameronetal}.

In \cite{Neumann}, Neumann asked whether there exists a primitive permutation group having an infinite suborbit which is paired with a finite suborbit. This amounts to asking whether there is a digraph with infinite in-valency and finite out-valency whose automorphism group is transitive on edges and primitive on vertices. Countable digraphs of this sort were constructed in \cite{Evans2} using amalgamation methods developed in model theory (cf. \cite{Cameron1} for background on such methods). In these examples, the descendant sets are directed trees, and the resulting examples are also highly arc transitive. Similar methods were used in \cite{EmmsEvans} to construct continuum-many non-isomorphic countable, primitive, highly arc transitive digraphs all with isomorphic descendant sets. So this  suggests that a classification of such digraphs is out of the question, even under the very strong assumption of high arc transitivity. Nevertheless,  Neumann (private communication) suggested that a classification of the \textit{descendant sets} in these digraphs might be possible, at least under stronger hypotheses on the automorphism group of the digraph.

Descendant sets in highly arc transitive digraphs of finite out-valency were studied by the first Author in \cite{amato, Amato}, following on from results obtained by M\"oller for locally finite, highly arc transitive digraphs in \cite{Moller4}.  This work isolates a small number of quite simple properties (essentially P0, P1, P3 of Section \ref{primitivesection} here) satisfied by such descendant sets and shows that these properties have rather strong structural consequences. In particular, the descendant set admits a non-trivial, finite-to-one homomorphism onto a tree. Digraphs having the given properties, but which are not trees are constructed in \cite{Amato, Moller4}. Moreover, (imprimitive) highly arc transitive digraphs having these as descendant sets are constructed in \cite{Amato, AT, AET}. 

It was conjectured in \cite{Amato}  that there are only countably many directed graphs with the properties for a descendant set isolated in \cite{Amato}. In Section \ref{structuresection} we reprove some of the results of \cite{Amato} in a slightly wider context and prove the conjecture. In particular, we have the following (note that a highly arc transitive digraph is distance transitive).

\begin{theorem}\label{mainresultA} Suppose that $D$ is a distance transitive digraph of finite out-valency. Assume either that $D$ has infinite in-valency, or that it has no directed cycles. Let $\Gamma = \Gamma_D$ be the descendant set of $D$. Then there are natural numbers $k(\Gamma)$ and  $M(\Gamma)$ with the property that if $\Gamma^{\leq M(\Gamma)}$ denotes the induced subdigraph on the set of vertices in $\Gamma$ which can be reached from the root by a directed path of length at most $M(\Gamma)$, then $\Gamma$ is determined up to isomorphism by $k(\Gamma)$, $M(\Gamma)$ and the finite digraph $\Gamma^{\leq M(\Gamma)}$. 
\end{theorem}

The proof of this is given at the end of Section \ref{structuresection}. Together with Corollary 4.4 of \cite{AET}, it gives a reasonable picture of the descendant sets in distance transitive digraphs of finite out-valency and infinite in-valency: conditions P0, P1, P3 are necessary and sufficient conditions for a digraph to be a descendant set in such a digraph, and there are only countably many digraphs satisfying these conditions. 

In Section \ref{primitivesection} we are interested in descendant sets under the additional assumption of  primitivity. Then main result is:

\begin{theorem} \label{mainresultB} Suppose that a digraph $\Gamma$ of finite out-valency satisfies conditions P0, P1, P2, P3. Then there is a countable primitive digraph $D_\Gamma$ of infinite in-valency with  descendant set $\Gamma$.
\end{theorem} 

The construction of $D_\Gamma$ is as in the paper \cite{Evans2}, where the descendant set $\Gamma$ is a tree. However, the proof of primitivity in the general case is much harder than in \cite{Evans2}, and this is where the novelty lies in the above result. We do not know whether the condition P2 on $\Gamma$ is a necessary condition here, but it is satisfied by all of the examples constructed in Section 5 of \cite{Amato}. So this gives new examples of descendant sets in primitive (and even highly arc transitive) digraphs of finite out-valency and infinite in-valency.  It would of course be interesting to find necessary and sufficient conditions on the descendant set in a primitive distance transitive digraph of finite out-valency and infinite in-valency. It would be even more interesting to know whether anything can be said without the assumption of distance transitivity.

\subsection{Notation and Terminology}

A digraph $(D; E(D))$ consists of a set $D$ of vertices, and a set $E(D)\subseteq D\times D$ of ordered pairs of vertices, the (directed) edges. Our digraphs will have no loops and no multiple edges. We will think of a  subset $X$ of the set $D$ of vertices as a digraph in its own right by considering the full induced subdigraph on $X$ (so $E(X) = E(D) \cap X^2$).  Throughout this paper, `subdigraph' will mean `full induced subdigraph'. Thus henceforth, we will not usually distinguish notationally between a digraph and its vertex set. In particular, we will usually refer to the digraph $(D; E(D))$ simply as `the digraph $D$'. Note that this is a different convention from the usual notation $D = (VD; ED)$. Furthermore, we will use notation such as `$\alpha \in D$' to indicate that $\alpha$ is a vertex of the digraph $D$.

We denote the automorphism group of the digraph $D$ by $\Aut(D)$. We say that $D$ is transitive (respectively, edge transitive)  if this is transitive on $D$ (respectively, $E(D)$). We say that $D$ is \textit{primitive} if $\Aut(D)$ is primitive on $D$, that is, there are no non-trivial $\Aut(D)$-invariant equivalence relations on $D$.

The \textit{out-valency} of a vertex $\alpha \in D$ is the size of the set $\{u\in D : (\alpha,u)\in E(D)\}$ of out-vertices of $\alpha$;  similarly, the \textit{in-valency} of $\alpha$ is the size of the set $\{u\in D : (u,\alpha)\in E(D)\}$ of in-vertices. 
Let $s\geq0$ be an integer. An \textit{$s$-arc} from $u$ to $v$ in $D$ is a sequence $u_{0}u_{1} \ldots u_{s}$ of $s+1$ vertices such that $u_0 = u$, $u_{s} = v$ and $(u_{i},u_{i+1})\in ED$ for $0\leq i<s$ and $u_{i-1}\neq u_{i+1}$ for $0<i<s$. Usually our digraphs will be asymmetric, in which case this last condition is redundant. We denote by $D^s(u)$ the set of vertices of $D$ which are reachable by an $s$-arc from $u$. The \textit{descendant set} $D(u)$ (or $\desc(u)$) of $u$ is $\bigcup_{s\geq 0} D^s(u)$.  Similarly the set $\anc(u)$ of \textit{ancestors} of $u$  is the set of vertices of which $u$ is a descendant.

 In particular,  fix $\alpha \in D$, and let $\Gamma = D(\alpha)$. If ${\rm Aut}(D)$ is transitive on the set of vertices of $D$, then $D(u)\cong \Gamma $ for all vertices $u$, and we shall speak of the digraph $\Gamma$ as \textit{the descendant set} of $D$.
 
 We say that the digraph $D$ is \textit{highly arc transitive} if for each $s \geq 0$, $Aut(D)$ is transitive on the set of $s$-arcs in $D$. Following \cite{Lam}, we say that a digraph $D$ is \textit{(directed)-distance transitive} if for every $s \geq 0$, $\Aut(D)$ is transitive on pairs $(u,v)$ for which there is an $s$-arc from $u$ to $v$, but no $t$-arc for $t <s$. Note that this implies vertex and edge transitivity, but is weaker than being highly arc transitive. We generally exclude the case of null digraphs, where there are no edges.

Henceforth, we shall be interested in the structure of a descendant set $\Gamma = \Gamma(\alpha)$ of a vertex $\alpha$ in some transitive digraph $D$ with finite out-valency $m$. We will be considering this as a digraph with its full induced structure from $D$. We refer to $\alpha$ as the \textit{root} of $\Gamma$ and write $\Gamma = \Gamma(\alpha)$ to indicate that any vertex of $\Gamma$ is a descendant of $\alpha$. Similarly, we write $\Gamma^i$ instead of $\Gamma^i(\alpha)$ for the set of vertices reachable by an $i$-arc starting at $\alpha$ and  if $\beta \in \Gamma(\alpha)$, then  we write $\Gamma(\beta) = \desc(\beta) \subseteq \Gamma(\alpha)$. It is clear that if $D$ is highly arc transitive, then $\Aut(\Gamma(\alpha))$ is transitive on $s$-arcs in $\Gamma(\alpha)$ which start at $\alpha$. Similarly, if $D$ is distance transitive, then $\Aut(\Gamma(\alpha))$ is transitive on $\Gamma^n(\alpha)$ for each $n \in \N$.

\section{The structure of descendant sets} \label{structuresection}

\subsection{Preliminaries}\label{prelim}
We work with digraphs $\Gamma$ having the following properties:
\begin{enumerate}
\item[\textbf{G0}]$\Gamma = \Gamma(\alpha)$ is a rooted digraph with finite out-valency $m > 0$ and $\Gamma^s(\alpha) \cap \Gamma^t(\alpha) = \varnothing$ whenever $s\neq t$.

\item[\textbf{G1}]$\Gamma(u)\cong \Gamma $ for all $u\in \Gamma$.

\item[\textbf{G2}] For $n\in \mathbb{N}$ we have $\vert \Gamma^n(\alpha) \vert < \vert \Gamma^{n+1}(\alpha)\vert$.

\item[\textbf{G3}] There is an integer $k \geq 1$ such that if $\l \geq k$ and $x \in \Gamma^\l(\alpha)$ and $z \in \Gamma(x)$, then $\anc(z) \cap \Gamma^1(\alpha) = \anc(x) \cap \Gamma^1(\alpha)$. \end{enumerate}

We shall see that conditions G0, G1, G3  hold when $\Gamma$ is the descendant set in a distance transitive digraph of finite out-valency and infinite in-valency (Corollary \ref{15}). The minimum possible $k$ in G3 is the parameter $k(\Gamma)$ which appears in Theorem \ref{mainresultA}. If $k(\Gamma) = 1$ then $\Gamma$ is a directed tree, however  Section 5 of \cite{Amato} constructs digraphs $\Gamma(\Sigma, t)$ satisfying G0 - G3 with arbitrary value for $k(\Gamma(\Sigma, t))$.

\textit{A priori} there could be continuum-many isomorphism types of digraphs with these properties. Our main result in this section (Theorem \ref{mainresult}) is that there are only countably many isomorphism types of digraph $\Gamma$ which satisfy G0, G1 and G3. To establish this, we show  that there is a natural equivalence relation $\rho$ on $\Gamma$  (refining the `layering' of $\Gamma$ given by G0) such that the quotient digraph $\Gamma/\rho$ is a directed tree. If G2 holds then this is not a directed line and the size of the layers $\Gamma^n$ grows exponentially.

\begin{lemma} \label{310} Suppose $D$ is a (non-null) digraph of finite out-valency which  has no directed cycles and is distance transitive. Then any descendant set $\Gamma(\alpha)$ in $D$ satisfies G0.\end{lemma}

\begin{proof} This is the same as the proof of Proposition 3.10 in \cite{Cameronetal}, so we omit the details. 
\end{proof}

\begin{lemma} \label{nocycles} Suppose $D$ is a digraph of finite out-valency with a directed cycle and whose automorphism group is either primitive on vertices or transitive on edges. Then $D$ has finite in-valency.
\end{lemma}

\begin{proof} First, suppose that $D$ is edge-transitive. Then there is a $K$ such that every edge of $D$ is in a directed $K$-cycle. Let $\alpha \in D$. Then every in-vertex $\beta$ of $\alpha$ is in $D^{K-1}(\alpha)$. But this set is finite, as $D$ has finite out-valency.

Now suppose $D$ is vertex-primitive. Consider the relation $\sim$ on $D$ given by $u\sim v \Leftrightarrow u \in D(v) \mbox{ and } v \in D(u)$. This is an $\Aut(D)$-invariant equivalence relation on $D$ and as $D$ contains a directed cycle, its classes are not singletons. Thus, by primitivity $u \sim v$ for all $u, v \in D$. In particular, every edge of $D$ is contained in a cycle. We can then argue as in the first case.
\end{proof}

\begin{lemma} \label{23A}  Suppose $\Gamma$ satisfies G0, G1 and that for each $i \in \N$ the automorphism group $\Aut(\Gamma)$ is transitive on $\Gamma^i$. Then $\Gamma$ satisfies G3.\end{lemma}

\begin{proof}
For $x \in \Gamma^i$, let $t_i = \vert\anc(x) \cap \Gamma^1\vert$. By the transitivity assumption, this  depends only on $i$. As $\anc(x) \cap \Gamma^1 \subseteq \anc(z) \cap \Gamma^1$ when $z \in \Gamma(x)$, we have  $t_1 \leq t_2\leq t_3 \leq \ldots \leq m$. Choosing $k$ so that $t_k$ is as large as possible, the result follows.\end{proof}

\begin{remark} \rm Note that in the above if G2 also holds, then $t_i < m$. Otherwise, for $\beta \in \Gamma^1$ we have $\Gamma^{i-1}(\beta) = \Gamma^i(\alpha)$ and so $\vert \Gamma^{i-1} \vert = \vert \Gamma^i\vert$ (by G1), contradicting G2.
\end{remark}

\begin{corollary}\label{15}
Suppose $D$ is a distance transitive digraph of finite out-valency $m > 0$ and is either of infinite in-valency, or has no directed cycles. Then the descendant set $\Gamma$  in $D$ satisfies G0, G1, G3. If the automorphism group of $D$ is also primitive on vertices, then $m >1$ and $\Gamma$ satisfies G2.
\end{corollary}

\begin{proof}
By Lemma \ref{nocycles}, if $D$ has infinite in-valency then $D$ has no directed cycles, so by Lemma \ref{310}, $\Gamma$ satisfies G0. As $D$ has transitive automorphism group, G1 holds. Distance transitivity implies that $\Aut(\Gamma)$ is transitive on each $\Gamma^i$, so G3 holds. 

Suppose $\Aut(D)$ is primitive on vertices of $D$. If G2 does not hold for some $n$, then for $\beta, \beta' \in \Gamma^1(\alpha)$ we have $\Gamma^n(\beta) = \Gamma^{n+1}(\alpha) = \Gamma^n(\beta')$. If $m >1$, then this gives a non-trivial equivalence relation on the vertices of $D$ which is preserved by $\Aut(D)$, and we have a contradiction to primitivity. So it remains to show that $m > 1$. But if $m=1$, then the underlying (undirected) graph of $D$ has no cycles. This contradicts primitivity of $\Aut(D)$, as it implies that being at even distance in the underlying graph is an equivalence relation on the vertices. 
\end{proof}

\subsection{Structure theory}

Throughout this section we assume that $\Gamma$ satisfies G0, G1, G3. We let $k$ be an  integer satisfying the condition in G3. The proofs in this section are all adapted from \cite{Amato}.
\medskip

\begin{lemma}\label{lemma1}
Suppose $n$ is a non-negative integer, $\beta \in \Gamma^n(\alpha)$, $\l \geq k$, $x \in \Gamma^{n+\l}(\alpha)$ and $z \in \Gamma(x) \cap \Gamma(\beta)$. Then $x \in \Gamma^{\ell}(\beta)$. 
\end{lemma}

\begin{proof}
This is by induction on $n$. The case $n = 0$ is trivial as then $\beta = \alpha$. In general let $\gamma \in \Gamma^{n-1}(\alpha)$ be an ancestor of $\beta$. By induction hypothesis, $x \in \Gamma^{\l+1}(\gamma)$. Now work with $\Gamma(\gamma) \cong \Gamma$ (by G1). As $\l \geq k$ and $z \in \Gamma(x)$ we have $\anc(z) \cap \Gamma^1(\gamma) = \anc(x) \cap \Gamma^1(\gamma)$ (by G3 in $\Gamma(\gamma)$). So $\beta \in \anc(x)$, that is $x \in \Gamma(\beta)$. As $\beta \in \Gamma^n(\alpha)$ and $x \in \Gamma^{n+\ell}(\alpha)$, it follows from G0 that $x \in \Gamma^\ell(\beta)$, as required.
\end{proof}
\begin{definition}\rm 
\begin{enumerate}
\item Suppose $\beta \in \Gamma$,  $x \in\Gamma^n(\beta)$ and $s \leq n$. Define 
\[ \Gamma_\beta^{-s}(x) = \{ w \in \Gamma^{n-s}(\beta): x \in \Gamma(w)\}.\]
\item For $\l \geq k$ and $x, y \in \Gamma^\l(\alpha)$ write $\rho(x,y)$ iff 
\[\Gamma_\alpha^{-k+1}(x) = \Gamma_\alpha^{-k+1}(y).\]
(Say that $\rho(x,y)$ does not hold in all other cases.)
\end{enumerate}
\end{definition}

So for $x, y \in \Gamma^\l(\alpha)$ we have that $\rho(x,y)$ holds iff $x,y$ have the same ancestors in $\Gamma^{\l-k+1}(\alpha)$. Clearly $\rho$ is an $\Aut(\Gamma)$-invariant equivalence relation on $\bigcup_{\l \geq k} \Gamma^\l$.

\begin{lemma}\label{lemma2}
Suppose $\l \geq k$ and $x, y \in \Gamma^\l(\alpha)$. If $\Gamma(x)\cap \Gamma(y) \neq \varnothing$, then $\rho(x,y)$ holds.
\end{lemma}

\begin{proof}
Note that the result holds for $\l = k$ by G3.

Suppose $\l = n+k$ with $n \geq 1$ and  that $z \in \Gamma(x) \cap \Gamma(y)$. Let $B = \{ \beta\in \Gamma^n(\alpha) : z \in \Gamma(\beta)\}$. If $\beta \in B$, then by Lemma \ref{lemma1}, $x,y \in \Gamma^k(\beta)$. Thus (by the case $\l = k$ in $\Gamma(\beta)$) we have $\anc(x) \cap\Gamma^1(\beta) = \anc(y)\cap \Gamma^1(\beta)$. But $\Gamma_\alpha^{-k+1}(x), \Gamma_\alpha^{-k+1}(y) \subseteq \bigcup_{\beta\in B} \Gamma^1(\beta)$. Thus $\Gamma_\alpha^{-k+1}(x) =  \Gamma_\alpha^{-k+1}(y)$, so $\rho(x,y)$.
\end{proof}

For $\l \geq k$ and $x \in \Gamma^l(\alpha)$ we write $[x]_\rho$ for the $\rho$-equivalence class containing $x$. We use notation such as $\v$, $\w$ etc. for such classes and write $\Gamma(\u) = \bigcup_{x \in \u} \Gamma(x)$ and $\Gamma^s(\u) = \bigcup_{x \in \u} \Gamma^s(x)$.

\begin{lemma}\label{lemma4}
Suppose $\l \geq k$ and $\v \subseteq \Gamma^\l(\alpha)$ is a $\rho$-class. Let $w \in \Gamma(\v)$. Then $[w]_\rho \subseteq \Gamma(\v)$.
\end{lemma}

\begin{proof}
It suffices to prove this when $w \in \Gamma^{\l+1}(\alpha)$. So suppose that $(v,w), (v', w')$ are directed edges and $\rho(w,w')$ holds. We need to show that $\rho(v,v')$ holds. Let $A = \Gamma_\alpha^{-1}(w)$ and $A' = \Gamma_\alpha^{-1}(w')$. By Lemma \ref{lemma2}, $A \subseteq [v]_\rho$ and $A' \subseteq [v']_\rho$. By definition,
$\Gamma_\alpha^{-k+1}(w) = \bigcup_{a \in A}\Gamma_\alpha^{-k+2}(a)$ and $\Gamma_\alpha^{-k+1}(w') = \bigcup_{a' \in A'}\Gamma_\alpha^{-k+2}(a')$. So  $\bigcup_{a \in A}\Gamma_\alpha^{-k+2}(a) = \bigcup_{a' \in A'}\Gamma_\alpha^{-k+2}(a')$, as $\rho(w,w')$ holds. It follows (by taking ancestors one level back) that $\bigcup_{a \in A}\Gamma_\alpha^{-k+1}(a) = \bigcup_{a' \in A'}\Gamma_\alpha^{-k+1}(a')$. But as $A \subseteq [v]_\rho$, the left hand side is equal to $\Gamma_\alpha^{-k+1}(v)$ and similarly the right hand side is equal to  $\Gamma_\alpha^{-k+1}(v')$. Thus $\rho(v,v')$ holds.
\end{proof}

\begin{corollary} \label{lemma65} Suppose $\l \geq k$ and $v \in \Gamma^\l(\alpha)$. Let $\v$ be the $\rho$-class containing $v$. Then the quotient digraph $\Gamma(\v)/\rho$ is a  rooted directed tree with finite out-valencies.\end{corollary}

\begin{proof}
The statement follows from Lemmas \ref{lemma2} and \ref{lemma4}. 
\end{proof}

Note that for $\beta \in \Gamma(\alpha)$ we can consider the equivalence relation $\rho$ computed in both $\Gamma(\alpha)$ and $\Gamma(\beta)$, where in the latter we only consider ancestors in $\Gamma(\beta)$ when defining $\rho$: \textit{a priori} this gives a coarser relation.

\begin{lemma}\label{lemma3}
Suppose $\beta \in \Gamma^n(\alpha)$ and $x \in \Gamma^\l(\beta)$ with $\l \geq 2k-1$. Then the $\rho$-class containing $x$ is the same whether it is computed in $\Gamma(\alpha)$ or $\Gamma(\beta)$.
\end{lemma}

\begin{proof} Note that $x \in \Gamma^{n+\l}(\alpha)$. First observe that if $y \in [x]_\rho$ (computed in $\Gamma(\alpha)$) then $x, y$ have the same ancestors in $\Gamma^{n+\l-k+1}(\alpha)$ and so also in $\Gamma^{n}(\alpha)$: in particular $y \in \Gamma(\beta)$. So to prove the statement, it suffices to show that $\Gamma_\beta^{-k+1}(x) = \Gamma_\alpha^{-k+1}(x)$. It is clear from the definition  that $\Gamma_\beta^{-k+1}(x) \subseteq \Gamma_\alpha^{-k+1}(x)$. Conversely, suppose $w \in \Gamma_\alpha^{-k+1}(x)$. Then $w \in \Gamma^{n+\l-k+1}(\alpha)$ and by assumption $n+\l -k+1 \geq n+k$. So by Lemma \ref{lemma1} we have $w \in \Gamma(\beta)$ and therefore $w \in \Gamma_\beta^{-k+1}(x)$. 
\end{proof}

Let $\l \geq 2k-1$ and let $\v$ be a $\rho$-class in $\Gamma^\l(\alpha)$. Let $T(\v)$ be the structure consisting of the induced digraph on  $\Gamma(\v)$  together with the equivalence relation induced by $\rho$ (coming from $\Gamma(\alpha)$). Recall that by  Lemma \ref{lemma4}, $T(\v)$ is a union of $\rho$-classes in $\Gamma(\alpha)$.  If $\w$ is another $\rho$-class (in $\bigcup_{\l \geq 2k-1} \Gamma^l(\alpha)$) then by a \textit{$\rho$-isomorphism} between $T(\v)$ and $T(\w)$ we mean a digraph isomorphism which respects $\rho$.

\begin{corollary}\label{lemma5}
Suppose $\v$ is a $\rho$-class in $\Gamma^\l(\alpha)$ with $l \geq 2k-1$. Then there is a $\rho$-class $\w$ in $\Gamma^{2k-1}(\alpha)$ and a $\rho$-isomorphism from $T(\w)$ to $T(\v)$.
\end{corollary}

\begin{proof}
Let $v \in \v$ and let $\beta \in \Gamma^{\l-2k+1}(\alpha)$ be an ancestor of $v$. So $v \in \Gamma^{2k-1}(\beta)$ and by Lemma \ref{lemma3} it follows that $\v \subseteq \Gamma(\beta)$. So $T(\v) \subseteq \Gamma(\beta)$ and the $\rho$-structure on $T(\v)$ is the same whether it is computed in $\Gamma(\alpha)$ or $\Gamma(\beta)$. By G1 there is a digraph isomorphism from $\Gamma(\alpha)$ to $\Gamma(\beta)$, and this induces a $\rho$-isomorphism between $T(\w)$, for some $\rho$-class $\w \subseteq \Gamma^{2k-1}(\alpha)$, and $T(\v) \subseteq \Gamma^{2k-1}(\beta)$, as required.
\end{proof}

Thus to any digraph $\Gamma$ satisfying G0, G1, G3, there are associated a  finite number of $\rho$-isomorphism types of $T(\v)$. In particular, we can refine Corollary \ref{lemma65} to:

\begin{corollary} \label{lemma652} Suppose $\l \geq 2k-1$ and $\v \subseteq \Gamma^\l(\alpha)$ is a $\rho$-class. Then the quotient digraph $T(\v)/\rho$ is a  rooted directed tree with a finite number of out-valencies.\hfill$\Box$\end{corollary}

\subsection{Counting isomorphism types}
We let $\T$ be the class of structures $T$ with the following properties
\begin{itemize}
\item $T$ is a digraph of finite out-valency and $T = T(\u)$ for some finite set $\u \subseteq T$.
\item $T^s(\u) \cap T^t(\u) = \varnothing$ whenever $s\neq t$.
\item There is an equivalence relation $\rho$ on $T$ such that each $\rho$-class is contained in a  layer $T^s(\u)$.
\item The quotient digraph $T/\rho$ is a directed forest.
\item For every $\rho$-class $\w$ there is a $\rho$-class $\v \subseteq \u$ and a $\rho$-isomorphism between $T(\v)$ and $T(\w)$.
\end{itemize}

We show:

\begin{theorem}\label{thm1}
There are only countably many $\rho$-isomorphism types of structures in $\T$.
\end{theorem}

\begin{corollary}\label{mainresult}
There are only countably many isomorphism types of digraph $\Gamma$ which have properties G0, G1 and G3.
\end{corollary}

\noindent\textit{Proof of Corollary.\/} Fix such a $\Gamma$. Let $T$ be the disjoint union of  digraphs $T(\v)$ with the equivalence relation $\rho$  as in the previous section, taking $\v$ to be a $\rho$-class in $\Gamma^{2k-1}$. So in fact, $T = \bigcup_{\l \geq 2k-1} \Gamma^{\l}$. Then $T \in\T$, by Corollaries \ref{lemma652} and \ref{lemma5}. Moreover we can recover $\Gamma$ from $T$ by looking at the descendant set of any vertex in $T$. Thus there are only countably many possibilities for $\Gamma$, by the above Theorem. \hfill $\Box$

\medskip

We now prove Theorem \ref{thm1}. Let $T= T(\u) \in \T$ and let $\v_1,\ldots, \v_r$ be the $\rho$-classes in $T^0 = \u$. We colour a $\rho$-class $\v$ in $T$ with colour $C_i$ if $i$ is (as small as possible) such that $T(\v)$ is $\rho$-isomorphic to $T(\v_i)$. If $d \in \N$, then we denote by $B_T^d$ the digraph on $\bigcup_{s \leq d} T^s$ together with the structure given by the $\rho$-classes and the colouring on this set. Similarly if $\v$ is a $\rho$-class we denote by $B_T^d(\v)$ the corresponding structure on $\bigcup_{s\leq d} T^s(\v)$. 

In the following, by a \textit{$\rho-C$-isomorphism}  we mean a digraph isomorphism which preserves the relation $\rho$ and the colouring.

\begin{lemma}\label{lemmaN}
For $T \in \T$ there is a natural number $N= N_T$ with the property that if $d \geq N$, $\v, \v'$ are $\rho$-classes in $T$ and $\alpha' : B^d_T(\v) \to B^d_T(\v')$ is a $\rho-C$-isomorphism, then there is a $\rho-C$-isomorphism $\alpha : T(\v) \to T(\v')$ with $\alpha(x) = \alpha'(x)$ for all $x \in \v$.
\end{lemma}

\begin{proof}
Let $A_0$ be the group of permutations induced on $T^0$ by $\rho-C$-automorphisms of $T$ which fix each $\rho$-class in $T^0$. Similarly for $d \geq 1$ let $A_d$ be the group of permutations induced on $T^0$ by $\rho-C$-automorphisms of $B^d_T$ which fix each $\rho$-class in $T^0$. Then $A_d \geq A_{d+1}$ and $A_0 = \bigcap_d A_d$, so there is a smallest integer $N\geq 1$ with $A_N = A_0$. In particular, for any $\rho$-class $\v$ in $T^0$, and $d \geq N$, any permutation of $\v$ which extends to a $\rho-C$-automorphism of $B^d_T(\v)$ extends to an automorphism of $T(\v)$. The same is therefore true for any $\rho$-class in $T$.

We show that this $N$ has the required property. So let $\v, \v'$ etc be as in the statement.  As $\v$, $\v'$ have the same colour, there is some $\rho-C$-isomorphism $\beta: T(\v) \to T(\v')$. Let $\beta'$ be its restriction to $B_T^d(\v)$. Then $\alpha'$, $\beta'$ both have image $B_T^d(\v')$ and $\gamma' = (\beta')^{-1}\circ \alpha'$ is a $\rho-C$-automorphism of $B_T^d(\v)$. So as $d \geq N$ there is a $\rho-C$-automorphism $\gamma$ of $T(\v)$ which agrees with $\gamma'$ on $\v$. It is easy to check that $\alpha = \beta\circ \gamma$ is a $\rho-C$-isomorphism with the required properties.
\end{proof}

\begin{proposition}\label{mainprop}
Suppose $T, S \in \T$ and $d > N_S$. If there is a $\rho-C$-isomorphism from $B^d_T$ to $B^d_S$, then there is a $\rho-C$-isomorphism from $B^{d+1}_T$ to $B^{d+1}_S$.
\end{proposition}

\begin{proof}
Let $\Phi : B^d_T \to B^d_S$ be a $\rho-C$-isomorphism. Note that $d\geq 1$. Let $\v_1,\ldots, \v_s$ be the $\rho$-classes in $T^1$ and $\w_i = \Phi(\v_i)$. So $\w_1,\ldots, \w_s$ are the $\rho$-classes in $S^1$. For $i \in \{1,\ldots,s\}$ there is a $\rho$-class $\u_i$ in $T^0$ and a $\rho-C$-isomorphism $f_i : T(\u_i) \to T(\v_i)$. Let $\z_i = \Phi(\u_i)$ and $\alpha_i' : B^{d-1}_S(\z_i) \to B^{d-1}_S(\w_i)$ be given by 
\[\alpha_i'(y) = \Phi(f_i(\Phi^{-1}(y))).\]
So $\alpha_i'$ is a $\rho-C$-isomorphism. As $d-1 \geq N_S$ it follows by Lemma \ref{lemmaN} that there is a $\rho-C$-isomorphism $\alpha_i : S(\z_i) \to S(\w_i)$ which agrees with $\alpha'$ on $\z_i$. 

\smallskip

We define $\Psi : B^{d+1}_T \to B^{d+1}_S$ as follows. For $x \in T^0$ we let $\Psi(x) = \Phi(x)$. If $x \in B^{d+1}_T \setminus T^0$ then there is a unique $i \leq s$ with $x \in B^d_T(\v_i)$ and in this case we define
\[ \Psi(x) = \alpha_i(\Phi(f_i^{-1}(x))).\]

It is easy to see that $\Psi$ is a well-defined bijection between $B^{d+1}_T$ and $B^{d+1}_S$. As $f_i, \Phi$ and $\alpha_i$ all preserve $\rho$-classes and the colouring, the same is true of $\Psi$. So it remains to show that $\Psi$ preserves edges and non-edges.

First we show that if $x \in B^1_T$, then $\Psi(x) = \Phi(x)$. If $x \in T^0$ then this is by definition of $\Psi$. If $x \in T^1$ then $x \in \v_i$ for some unique $i \leq s$. So $f_i^{-1}(x) \in \u_i$ and $\Phi(f_i^{-1}(x)) \in \z_i$, whence 
\[\Psi(x) = \alpha_i\Phi f_i^{-1}(x) = \alpha_i'\Phi f_i^{-1}(x) = \Phi(x).\]
Thus $\Psi$ preserves edges and non-edges in $B^1_T$. 

If $x,y \in B^{d+1}_T\setminus T^0$ and $(x,y)$ is an edge, then $x,y \in B^d_T(\v_i)$ for some $i$. Then $\Psi(x) = \alpha_i\Phi f_i^{-1}(x)$ and $\Psi(y) = \alpha_i\Phi f_i^{-1}(y)$ and so, as $\alpha_i, \Phi$ and $f_i$ preserve edges, $(\Psi(x), \Psi(y))$ is an edge in $B^d_S$. By the same argument, if $x,y \in B^d_T(\v_i)$ and $(x,y)$ is a non-edge, then $(\Psi(x),\Psi(y))$ is a non-edge. Finally, if $x, y$ lie in different $B_T^d(\v_i)$ then $\Psi(x), \Psi(y)$ lie in different $B^d_S(\w_i)$, so $(\Psi(x),\Psi(y))$ is a non-edge.
\end{proof}

\begin{corollary}\label{isom} Suppose $T, S \in \T$ and $B_T^d$, $B_S^d$ are $\rho-C$-isomorphic for some $d\geq N_S$. Then $T$ and $S$ are $\rho-C$-isomorphic.
\end{corollary}

\begin{proof}
By assumption and Proposition \ref{mainprop}, for $n \geq N_S$ the set $I_n$ of $\rho-C$-isomorphisms $B^n_T \to B^n_S$ is non-empty. Restriction gives a map $I_{n+1} \to I_n$ and so, as each $I_n$ is finite, K\"onig's Lemma implies that there is a $\rho-C$-isomorphism $T \to S$. 
\end{proof}

\noindent\textit{Proof of Theorem \ref{thm1}.\/} Suppose $T \in \T$. As above, consider this with a colouring of the $\rho$-classes. Let $N=N_T$ be as in Lemma \ref{lemmaN}. Then by Corollary \ref{isom}, the (coloured) ball $B^{N+1}_T$ determines $T$ within $\T$ up to isomorphism. There are only countably many possibilities for this finite structure, hence the result.\hfill $\Box$

\medskip

We now give the proof of Theorem \ref{mainresultA}.

\medskip

\noindent\textit{Proof of Theorem \ref{mainresultA}.\/} Let $D$ and $\Gamma = \Gamma_D$ be as in the statement of the theorem. First, we define the numbers $k(\Gamma)$ and $M(\Gamma)$. By Corollary \ref{15}, $\Gamma$ satisfies G0, G1, G3. Let $k = k(\Gamma)$ be the smallest value of $k$ which satisfies G3 for $\Gamma$. 

Let $\w$ be a $\rho$-class in $\Gamma^{2k-1}$. By distance transitivity, $\Aut(\Gamma)$ is transitive on each $\Gamma^\ell$, so if $\ell \geq 2k-1$ and $\v$ is a $\rho$-class in $\Gamma^\ell$, then there is a $\rho$-isomorphism from $\Gamma(\w)$ to $\Gamma(\v)$, by Corollary  \ref{lemma5}. Let $T_\Gamma \in \T$ consist of $\Gamma(\w)$ together with its $\rho$-structure. Note that we have only one `colour' $C_i$ used here, so $\rho$-isomorphisms will be $\rho-C$-isomorphisms in what follows.

Let $M(\Gamma) = 2k(\Gamma) + N_{T_{\Gamma}}$, where $N_{T_\Gamma}$ is as in Lemma \ref{lemmaN} (chosen as small as possible). Thus, from the proof of Lemma \ref{lemmaN}, $N_{T_\Gamma}$ is the smallest value of $N$ such that any permutation of $\w$ which extends to a $\rho$-automorphism of $T_{\Gamma}^{\leq N}$ extends to a $\rho$-automorphism of $T_\Gamma$. 

Now suppose that $D_1$, $D_2$ are distance transitive digraphs satisfying the hypotheses of the theorem. Let $\Gamma_i = \Gamma_{D_i}$ and suppose that $k(\Gamma_1) = k(\Gamma_2) = k$, $M(\Gamma_1) = M(\Gamma_2) = M$ and $\theta : \Gamma_1^{\leq M} \to \Gamma_2^{\leq M}$ is an isomorphism. As $k(\Gamma_1) = k(\Gamma_2)$, $\theta$ gives a $\rho$-isomorphism $\Gamma_1^\ell \to \Gamma_2^\ell$ for $k \leq \ell \leq M$. 

Let $\w_i \in \Gamma_i^{2k-1}$ be  $\rho$-classes, with $\theta(\w_1) = \w_2$. Let $T_i = \Gamma_i(\w_i)$, considered also with its $\rho$-structure. Then $N_{T_1} = N_{T_2}= N$ and $\theta$ gives a $\rho$-isomorphism between the balls $B_{T_1}^{N+1}$ and  $B_{T_2}^{N+1}$. By Corollary \ref{isom} (and the above remark on colours) $T_1$ and $T_2$ are isomorphic. It then follows that $\Gamma_1$ and $\Gamma_2$ are isomorphic, as required. \hfill $\Box$

\section{Constructions} \label{primitivesection}

In this section, we prove Theorem \ref{mainresultB}. The construction of the digraphs $D_\Gamma$ is as in \cite{Evans2} and we recall briefly some notation and terminology from there. 

Suppose $D$ is a digraph and $A \subseteq D$. We write $A \leq D$ if for every $a \in A$ we have $\desc(a) \subseteq A$. We say that $A \leq D$ is \textit{finitely generated (f.g.)\/}  if there is a finite $X \subseteq A$ with $A = \bigcup_{a \in X} \desc(a)$, and say that $X$ is a \textit{generating set\/} for $A$. 

We write $A \leq^+ D$ if $A \leq D$ and
\begin{enumerate}
\item[(i)]  for every $b \in D$, if $\desc(b) \setminus A$ is finite, then $b \in A$;
\item[(ii)] for all $b \in D$, $\desc(b)\cap A$ is finitely generated.
\end{enumerate}
It is easy to check (cf. Lemma 2.2 of \cite{Evans2}) that if $A \leq^+ B \leq^+ C$ then $A\leq^+ C$ (and similarly for $\leq$).

\medskip

We work with a digraph $\Gamma$ having the following properties:
\begin{enumerate}
\item[\textbf{P0}]$\Gamma = \Gamma(\gamma)$ is a rooted digraph with finite out-valency $m > 0$ and $\Gamma^s(\gamma) \cap \Gamma^t(\gamma) = \emptyset$ whenever $s\neq t$.

\item[\textbf{P1}]$\Gamma(u)\cong \Gamma $ for all $u\in \Gamma$.

\item[\textbf{P2}] For all $a \in \Gamma$ we have $\desc(a) \leq^+ \Gamma$

\item[\textbf{P3}] For all natural numbers $n$,  $\Aut(\Gamma)$ is transitive on $\Gamma^n$. \end{enumerate}

Of course, P0 and P1 are the same as G0, G1 and P3 implies G3 (as in Lemma \ref{23A}). From Section \ref{prelim}, if $\Gamma$ is the descendant set of a vertex in an infinite, distance transitive digraph  $D$ of finite out-valency and with no directed cycles then $\Gamma$ satisfies P0, P1, P3. If $D$ is primitive, then (as noted in \cite{Amato} under the stronger assumption of high arc transitivity) $\Gamma$ satisfies the following weaker version of P2:

\begin{enumerate}
\item[\textbf{P2$'$}] For all  $a_1, a_2 \in \Gamma$, if $\Gamma(a_1)\setminus \Gamma(a_2)$ and $\Gamma(a_2)\setminus \Gamma(a_1)$ are finite, then $a_1 = a_2$.
\end{enumerate}

Note in particular that P2 implies that different vertices have different sets of out-vertices.

Section 5 of \cite{Amato} gives examples $\Gamma(\Sigma, k)$ which satisfy P0, P1, P2$'$, P3 and it can be checked that these examples also satisfy P2. In this section we prove that if $\Gamma$ satisfies P0-P3, then there is a primitive digraph $D_\Gamma$ with $\Gamma$ as its descendant set. If $\Gamma$ has the property that $\Aut(\Gamma)$ is transitive on $n$-arcs from $\gamma$ (as is the case with the $\Gamma(\Sigma, k)$ from \cite{Amato}), then the $D_\Gamma$ which we construct will be highly arc transitive.

The construction of $\DG$ is essentially the same Fra\"{\i}ss\'e amalgamation class construction which was used in \cite{Evans2}. We will recall this briefly, making use of results from \cite{AET}. Once we have $\DG$, the main work of the section will be in proving primitivity of $\Aut(\DG)$. 

So suppose $\Gamma$ satisfies P0-P3. Let $\CB_\Gamma$ consist of the  digraphs $A$ with the property that for every $a \in A$, $\desc(a) \leq^+ A$ and $\desc(a) \cong \Gamma$. Let $\C_\Gamma$ be the finitely generated elements of $\CB_\Gamma$. Note that $\Gamma \in \C_\Gamma$, so in particular, $\C_\Gamma$ is non-empty. 

If $A, B \in \CB_\Gamma$ a digraph embedding $f : A \to B$ is called a \textit{$\leq^+$-embedding} if $f(A) \leq^+ B$. We say that $\leq^+$-embeddings $f_i : A \to B_i$ (for $i  = 1, 2$) are \textit{isomorphic} if there is a digraph isomorphism $h : B_1 \to B_2$ with $f_2 = h\circ f_1$. 

\begin{lemma} (cf. 2.14 of \cite{Evans2}) \label{countable} Suppose $\Gamma$ satisfies P0-P3. Then
\begin{enumerate}
\item there are countably many isomorphism types of digraphs in $\C_\Gamma$;
\item if $A, B \in \C_\Gamma$ then there are countably many isomorphism types of $\leq^+$-embeddings $f : A \to B$.
\end{enumerate}
\end{lemma}

\begin{proof} This follows from results in Section 4 of \cite{AET}. The digraph $\Gamma$ satisfies the conditions T1, T2, T3, T4 in Theorem 4.3 of \cite{AET} (the first three are just P0, P1, P3 and T4 follows from these as in Remark 4 of \cite{AET}). As in the proof of Corollary 4.4 of \cite{AET}, it follows that $\Gamma$ satisfies conditions (C1), (C2) of Theorem 4.1 of \cite{AET}. Under these conditions, Lemma 4.2 of \cite{AET} gives the stronger result that there are countably many isomorphism types of digraph embeddings $f : A \to B$ with $A, B \in \C_\Gamma$ and $f(A) \leq B$. (Note that $\C_\Gamma$ as defined here is a subset of the $\C_\Gamma$ defined in Section 4.1 of \cite{AET}.) The result we want follows: for (1), take $A = \emptyset$ and (2) is immediate.
\end{proof}

It is easy to show that $\C_\Gamma$ is closed under free amalgamation over finitely generated $\leq^+$-subsets (as in Lemma 2.6 of \cite{Evans2}). More formally, if  $B_1, B_2 \in \C_\Gamma$ and $A \leq^+ B_i$ is f.g., then the digraph $F$ which has vertices the disjoint union of $B_1$ and $B_2$ over $A$ and whose edges are the edges of $B_1$ and $B_2$ is also in $\C_\Gamma$. Furthermore,  $B_1, B_2 \leq^+ F$. This gives the following $\leq^+$-amalgamation property:

\begin{lemma} \label{amal} Suppose $A, B_1, B_2 \in \C_\Gamma$ and $f_i : A \to B_i$ are $\leq^+$-embeddings (for $i = 1,2$). Then there exist $F \in \C_\Gamma$ and $\leq^+$-embeddings $g_i : B_i \to F$ with the property that $g_1(f_1(a)) = g_2(f_2(a))$ for all $a \in A$. \hfill $\Box$
\end{lemma}

Note that if $f_1$ is inclusion, then we can also take $g_1$ to be inclusion here.

Once we have these lemmas, the following existence and uniqueness result is fairly standard and we omit some of the details of the proof.

\begin{theorem}\label{11} There is a countable digraph $D_\Gamma$ with the properties:
\begin{enumerate}
\item If $a \in D_\Gamma$ then $\desc(a) \leq^+ D_\Gamma$ and $\desc(a)\cong \Gamma$.
\item If $X \subseteq D_\Gamma$ is finite, there is a f.g. $A\leq^+ D_\Gamma$  with $X \subseteq A \in \C_\Gamma$.
\item If $A \leq^+ D_\Gamma$ is f.g. and $f : A \to B \in \C_\Gamma$ is such that $f(A) \leq^+ B$ then there is $g : B \to D_\Gamma$ with $gf(a) = a$ for all $a \in A$ and $g(B) \leq^+ D_\Gamma$. 
\end{enumerate}
Moreover, $D_\Gamma$ is uniquely determined up to isomorphism by these conditions and is $\leq^+$-homogeneous, meaning that if $A_1, A_2 \leq^+ D_\Gamma$ are f.g. and $h : A_1\to A_2$ is an isomorphism, then $h$ extends to an automorphism of $D_\Gamma$.
\end{theorem}

\begin{proof} Note that (1) here follows from (2). For the existence part, we build a chain of digraphs $D_i \in \C_\Gamma$ 
\[ D_1 \leq^+ D_2 \leq^+ D_3 \leq^+ \ldots \]
with the property:
\begin{quote}{(*)} if $A \leq^+ D_i$ is finitely generated and $f : A \to B$ is a $\leq^+$-embedding with $B \in \C_\Gamma$, then there is some $j \geq i$ and a $\leq^+$-embedding $g : B \to D_j$ such that $g(f(a)) = a$ for all $a \in A$. \end{quote}
Once we have this, we let $D_\Gamma$ be the union $\bigcup_{n \in \N} D_n$. Then (2) follows as each $D_n$ is in $\C_\Gamma$, and (3) follows from (*). 

In order to obtain (*) we build the $D_n$ inductively. During this process, there will be countably many `tasks' to be performed: there are countably many choices of f.g. $A$ in each $D_i$ and countably many isomorphism types of $\leq^+$-embeddings $f : A \to B$ with $B \in \C_\Gamma$ (by Lemma \ref{countable}). As we have countably many steps available during the construction, it will suffice to show how to complete one of these tasks: ensuring that they are all completed during some stage of the construction is then just a matter of organisation (see the proof of Theorem 2.8 of \cite{EmmsEvans} for a  formal way of doing this). 

So suppose $D_n$ has been constructed, $A \leq^+ D_n$ is f.g. and $f : A \to B$ is a $\leq^+$-embedding with $B \in \C_\Gamma$. Using amalgamation (Lemma \ref{amal}) we can find $D_n \leq^+ D_{n+1}$ and a $\leq^+$-embedding $g : B \to D_{n+1}$ with $g(f(a)) = a$ for all $a \in A$, as required. 

This completes the construction of some countable digraph $D_\Gamma$ with properties (1), (2), (3). For the `Moreover' part, suppose $D'_\Gamma$ is also a countable digraph with properties (1), (2), (3). Suppose $A \leq^+ D_\Gamma$ and $A' \leq^+ D'_\Gamma$ are f.g. and $h: A \to A'$ is an isomorphism. It will suffice to prove that $h$ extends to an isomorphism $D_\Gamma \to D'_\Gamma$. As $D_\Gamma, D'_\Gamma$ are countable, this follows by a back-and-forth argument (and symmetry) once we show:

\textit{Claim:\/} If $c \in D_\Gamma$ there exist finitely generated $B \leq^+ D_\Gamma$ and $B' \leq^+ D'_\Gamma$ with $A \leq^+ B$, $A' \leq^+ B'$ and $c \in B$ and an isomorphism $g: B \to B'$ extending $h$.

Existence of $B$ here follows from (2) in $D_\Gamma$. Existence of $g$ and $B'$ then follows from (3) in $D'_\Gamma$ (applied to $h^{-1} : A' \to B$). 
\end{proof}

 It is clear that $\leq^+$-homogeneity together with Property (1) in Theorem \ref{11} imply that $\Aut(D_\Gamma)$ is transitive on vertices. Moreover,  for every $a \in D_\Gamma$, any automorphism of the descendent set $D_\Gamma(a)$ extends to an automorphism of $D_\Gamma$ (necessarily fixing $a$) and so by P3, $\Aut(D_\Gamma/a)$ (the stabilizer of $a$) is transitive on $D_\Gamma^n(a)$ (vertices reachable by an $n$-arc from $a$). Thus $D_\Gamma$ is distance transitive. 

The remainder of this section is devoted to showing:

\begin{theorem}\label{main}
With the above notation, $\Aut(D_\Gamma)$ is primitive on the vertices of $D_\Gamma$.
\end{theorem}

\noindent By the above remarks, Theorem  \ref{mainresultB} then follows. 

\medskip

The following lemma is a simple application of free amalgamation and the extension property (3) in Theorem \ref{11}, but we shall give the details.

\begin{lemma}\label{12new}
Suppose $A \leq^+ B \leq^+ D_\Gamma$ and $A, B$ are finitely generated. Suppose $h$ is an automorphism of $A$. Then $h$ can be extended to $g \in \Aut(D_\Gamma)$ so that $B \cap gB = A$.
\end{lemma}

\begin{proof} Let $B'$ be any set with $B \cap B' = A$ and $\vert B \setminus A\vert = \vert B'\setminus A\vert$. Extend $h$ to a bijection $s : B \to B'$. Let $E$ be the set of directed edges in $B$ and define a digraph relation $E'$ on $B'$ by $(x,y) \in E' \Leftrightarrow (s^{-1}x,s^{-1}y) \in E$. As $s$ restricted to $A$ is $h$ and this is a digraph isomorphism, we have $E \cap A^2 = E' \cap A^2$. Clearly $s$ is then a digraph isomorphism. Let $F = B \cup B'$ with digraph relation $E \cup E'$. So $F$ is the free amalgam of $B$ and $B'$ over $A$ and $B, B' \leq^+ F \in \C_\Gamma$. By the extension property (Theorem \ref{11}(3)) over $B$, we can assume that $F \leq^+ D_\Gamma$. Then $s : B \to B'$ extends to an automorphism $g$ of $D_\Gamma$ (by the `Moreover' in Theorem \ref{11}), and this has the required properties.
\end{proof}

 If $a, b \in \DG$ and $a\neq b$, let $\Delta(a,b)$ be the orbital digraph which has $(a,b)$ as an edge. So this is the digraph with vertex set $\DG$ and edge set the $\Aut(\DG)$-orbit which contains $(a,b)$. By D. G. Higman's criterion, to prove the primitivity, it will be enough to show that each such $\Delta(a,b)$ is connected (meaning that its underlying undirected graph is connected).
 
 \begin{lemma}\label{13} If $(a,b)$ and $(a,b')$ are in the same $\Aut(\DG)$-orbit, and $\Delta(b,b')$ is connected, then $\Delta(a,b)$ is connected (of diameter at most twice that of $\Delta(b,b')$).
 \end{lemma}

\begin{proof} There is an (undirected) $\Delta(a,b)$-path $bab'$ from $b$ to $b'$. Thus if $(b_1,b_2)$ is an edge in $\Delta(b,b')$ there is a $\Delta(a,b)$-path of length 2 from $b_1$ to $b_2$. As $\Delta(b,b')$ is connected, if $x,y \in \DG$ there is a $\Delta(b,b')$-path from $x$ to $y$, and therefore there is a $\Delta(a,b)$-path from $x$ to $y$ with at most twice as many edges.
\end{proof}

\begin{lemma}\label{14} If $a, b \in D_\Gamma$ are distinct vertices and $\desc(a)\cap \desc(b) = \emptyset$, then $\Delta(a,b)$ is connected (of diameter at most 4).
\end{lemma}

\begin{proof} First suppose that $\desc(a)\cup\desc(b) \leq^+ \DG$. Given $x_1, x_2 \in \DG$, by Theorem \ref{11}(2) there is a finitely generated $Z \leq^+ \DG$ with $x_1, x_2 \in Z$. Let $X$ be the disjoint union of $Z$ and a copy $C$ of $\Gamma$. So $X$ is the free amalgam of $Z$ and $C$ over the empty set and $X \in  \C_\Gamma$. By Theorem \ref{11}(3) we may assume that $X \leq^+ \DG$. Let $c \in \DG$ be the generator of $C$. Then $\desc(c)\cap \desc(x_i) = \emptyset$ and $\desc(x_i) \cup \desc(c) \leq^+ \DG$ (for $i = 1,2$). It follows (using the `Moreover' part of Theorem \ref{11}) that $(c,x_1)$ and $(c,x_2)$ are edges in $\Delta(a,b)$. So $\Delta(a,b)$ has diameter 2. 

In general, there is $b' \in \DG$ such that $(a,b')$ is an edge in $\Delta(a,b)$ and f.g. $B, B' \leq^+ \DG$ with $a,b \in B$ and $a,b' \in B'$ satisfying $B\cup B' \leq^+ \DG$ and $B \cap B' = \desc(a)$ (again, this uses free amalgamation and property (3) of $\DG$; alternatively we can apply Lemma \ref{12new}). Then $b,b'$ are as in the first part of the proof and so $\Delta(b,b')$ has diameter 2. Therefore by Lemma \ref{13}, $\Delta(a,b)$ has diameter at most $4$. 
 \end{proof}
 
 The main work is in proving:
 
 \begin{proposition} \label{mainprop2} Suppose $a,b \in \DG$ are distinct vertices and $a\not\in \desc(b)$ and $b \not\in\desc(a)$. Then there exist $r \in \N$ and $g_1, \ldots, g_r \in \Aut(\DG)$ such that, setting $b_0 = a$, $b_1=b$ and $b_{i+1}=g_ib_i$ for $i \geq 1$, we have $g_i \in \Aut(\DG/b_{i-1})$ and $\desc(b_r)\cap \desc(b_{r+1}) = \emptyset$.
 \end{proposition}

We postpone the proof for now, and carry on with the proof of Theorem \ref{main}.

\begin{lemma}\label{16} Suppose $a,b \in \DG$ are distinct vertices and $a\not\in \desc(b)$ and $b \not\in\desc(a)$. Then $\Delta(a,b)$ is connected. 
\end{lemma}

\begin{proof} Let $r$,  $g_i$ and $b_i$ be as in Proposition \ref{mainprop2}. By Lemma \ref{14}, $\Delta(b_r, b_{r+1})$ is connected. So by Lemma \ref{13}, $\Delta(b_{r-1}, b_r)$ is connected. Proceeding in this way, we obtain that $\Delta(b_0, b_1)$ is connected.
\end{proof}

The remaining case to consider is:

\begin{lemma} \label{17} Suppose $a \in \desc(b)$ (and $a\neq b$). Then $\Delta(a,b)$ is connected.
\end{lemma}

\begin{proof} By Lemma \ref{12new}, there exists $g \in \Aut(\DG/a)$ with $\desc(b)\cap \desc(gb) = \desc(a)$. By Lemma \ref{16}, $\Delta(b,gb)$ is connected. So by Lemma \ref{13}, $\Delta(a,b)$ is connected.
\end{proof}

By D. G. Higman's criterion, Lemmas \ref{16} and \ref{17} establish Theorem \ref{main}. Thus, it remains to prove Proposition \ref{mainprop2}. We first recall some definitions and results from Section \ref{structuresection}.

We tend to write $\Gamma(a)$ for $\desc(a)$ in order to emphasise the isomorphism with $\Gamma$. We write $\Gamma^n(a)$ for the descendants reachable by an $n$-arc, and say that $b \in \Gamma^n(a)$ is \textit{at level $n$} with respect to $a$. Variations such as $\Gamma^{\geq n}(a)$ (for $\bigcup_{m\geq n} \Gamma^m(a)$) will also be used. 

For any $x \in \Gamma^\ell$ (with $\ell \geq 1$) consider the number of ancestors of $x$ in $\Gamma^1$. By P3 this depends only on $\ell$ and is non-decreasing as $\ell$ increases. Thus there is some value $k$ of $\ell$ for which it reaches a maximum size (which is necessarily less than $q$, the out-valency of $\Gamma$). Let $K = 2k-1$. 

Now we work in $\DG$. For $a, x,y \in \DG$ with $x,y \in \Gamma^\ell(a)$ and $\ell\geq k$, we write $\rho_a(x,y)$ to indicate that $x,y$ have the same ancestors in $\Gamma^{\ell-k+1}(a)$. So this is the relation $\rho$ on $\Gamma(a)$ used in Section \ref{structuresection}. This is an equivalence relation on $\Gamma^\ell(a)$ which is clearly invariant under the stabiliser $\Aut(\DG/a)$. Denote the equivalence class containing $x$ by $\rho_a[x]$.

\begin{lemma}\label{18}\begin{enumerate}
\item If $x \in \Gamma(a)$ and $y \in \Gamma^{\geq K}(x)$ and $\rho_a(z,y)$, then $z \in \Gamma(x)$.
\item If $x \in \Gamma^m(a)$ and $\ell \geq K$, then $\Gamma(x)\cap \Gamma^{m+\ell}(a)$ is a union of $\rho_a$-classes.
\item If $x,y \in \Gamma^\ell(a)$ and $\ell \geq k$ and $\Gamma(x)\cap \Gamma(y) \neq \emptyset$ then $\rho_a(x,y)$.
\end{enumerate}
\end{lemma}

\begin{proof} (1) This follows directly from Lemma \ref{lemma3}.

(2) This follows immediately from (1).

(3) This is Lemma \ref{lemma2}.
\end{proof}

Let $\ell \geq k$ and $a \in \DG$. For $x,y \in \Gamma^\ell(a)$ write $\sigma_a^0(x,y)$ if $\Gamma(x) \cap \Gamma(y)\neq \emptyset$ and let $\sigma_a$ be the transitive closure of this relation. Note that this is an equivalence relation (on each $\Gamma^\ell(a)$) which is clearly $\Aut(\DG/a)$-invariant.  Moreover, by (3) above, $\sigma_a(x,z)$ implies $\rho_a(x,z)$. Write $\sigma_a[x]$ for the $\sigma_a$-class containing $x$.

\begin{remark}\rm In the examples in \cite{Amato}, we have $\rho_a = \sigma_a$. It is not clear whether this holds for arbitrary $\Gamma$ satisfying P0-P3.
\end{remark}

\begin{lemma} \label{110} Suppose $b, b' \in \DG$ are such that if $x \in X = \Gamma(b)\cap \Gamma(b')$ then $x\in \Gamma^\ell(b) \Leftrightarrow x \in \Gamma^\ell(b')$ (so points in the intersection are at the same `level' with respect to $b$ and $b'$).
 Then there exists $n \geq K$  such that $X = \desc(X \cap \Gamma^{\leq n-K}(b))$. Moreover, for such an $n$, we have $X \cap \Gamma^n(b) = X \cap \Gamma^n(b')$ and if $y \in X\cap \Gamma^n(b)$, then $\sigma_b[y] = \sigma_{b'}[y]$. 
\end{lemma}

\begin{proof} As $X$ is finitely generated, it will suffice to take $n$ large enough so that $\Gamma^{\leq n-K}$ contains a generating set for $X$. For the rest, note first  that if $y \in X \cap \Gamma^n(b)$, then by the opening assumption on levels, $y \in X\cap \Gamma^n(b')$. Moreover, the assumption on $n$ means that there is $x \in X$ such that $y \in \Gamma^{\geq K}(x)$. So by Lemma \ref{18}(2), $\rho_b[y] \subseteq X$. Thus $\sigma_b[y] \subseteq X$. Similarly $\sigma_{b'}[y] \subseteq X$. Now, if $\sigma^0_b(y,z)$ holds then $z \in X$ and so (by definition of $\sigma^0$) also $\sigma^0_{b'}(y,z)$ holds. The statement follows.
\end{proof}

\noindent\textit{Proof of Proposition \ref{mainprop2}.\/}  Let $X = \Gamma(a)\cap \Gamma(b)$. Clearly we may assume that $X$ is non-empty. Then $X \leq^+ \Gamma(a)$ (as $\Gamma(b)\leq^+ \DG$) and $X \neq \Gamma(a)$. By free amalgamation (over $\Gamma(a)$) there is a copy $b'$ of $b$ over $\Gamma(a)$ with $\Gamma(b)\cap \Gamma(b') = X$. More precisely, apply Lemma \ref{12new} with $A = \Gamma(a)$, $h$ the identity on $A$, and $B\leq^+ D_\Gamma$ finitely generated with $a, b \in B$. It follows that there is $g_1$ in the pointwise stabilizer $\Aut(\DG/\Gamma(a))$ of $\Gamma(a)$ such that $g_1(B) \cap B = \Gamma(a)$. Then, with $b' = g_1b$, we have $X \subseteq \Gamma(b)\cap \Gamma(b') \subseteq \Gamma(b) \cap B \cap g_1(B) = \Gamma(b) \cap \Gamma(a) = X$ as $g_1$ fixes all elements of $X$. Moreover, because $g_1$ fixes all elements of $X$ and sends $b$ to $b'$, the points in $X$ are at the same level with respect to $b$ and $b'$, as in Lemma \ref{110}. 

Let $b_0=a$, $b_1 = b$ and $b_2 = b'$. Suppose, for  $i\geq 2$, that  $g_i \in \Aut(\DG/ b_{i-1})$ and $b_{i+1} = g_ib_i$ is such that $\Gamma(b_{i+1}) \cap \Gamma(b_i) \subseteq \Gamma(b_{i-1})$ and $b_{i+1} \neq b_i$ (the existence of such $g_i$ will be proved later on.) We make a series of claims, refining the choice of the $g_i$ as we proceed. We may suppose $\Gamma(b_1) \cap \Gamma(b_2)\neq \emptyset$. 

\medskip

\textit{Claim 1:\/} For $i \geq 1$, if $x \in \Gamma(b_i) \cap \Gamma(b_{i+1})$, then $x$ is at the same level with respect to $b_i$ and $b_{i+1}$.

We prove this by induction, noting first that it holds for $i = 1$, by construction. By assumption on the $g_i$, we have  $x \in \Gamma(b_{i-1})\cap\Gamma(b_i)\cap \Gamma(b_{i+1})$. So by inductive assumption, $x$ has the  same level with respect to $b_{i-1}$ and $b_i$. But 
$\Gamma(b_{i-1})\cap \Gamma(b_{i+1})$ is $g_i(\Gamma(b_{i-1})\cap \Gamma(b_{i}))$, so a point in this intersection has the same level with respect to $b_{i-1}$ and $b_{i+1}$. Thus $x$ has the same level with respect to $b_{i-1}, b_i$ and $b_{i+1}$. \hfill ($\Box_{Claim\, 1}$)

\medskip

Choose $n \geq K$ such that $Y_1 = \Gamma(b_1) \cap \Gamma(b_2)$ is generated by its intersection with $\Gamma^{\leq n-K}(b_1)$. Let $Z_1 = Y_1 \cap \Gamma^n(b_1) = \Gamma^n(b_1)\cap \Gamma^n(b_2)$. By Lemma \ref{110} we have:

\medskip

\textit{Claim 2:\/} $Z_1$ is a union of sets which are simultaneously $\sigma_{b_1}$-classes and $\sigma_{b_2}$-classes. 

\medskip

\textit{Claim 3:\/} For $i \geq 2$, if $Z_i = \Gamma^n(b_i) \cap \Gamma^n(b_{i+1})$, then $Z_{i}$ is a union of sets which are simultaneously  $\sigma_{b_j}$-classes for all $1 \leq j \leq i+1$.

Note that Claim 1 and the fact that  $\Gamma(b_{i+1})\cap \Gamma(b_i) \subseteq \Gamma(b_{i-1})$ imply that $Z_i \subseteq Z_{i-1} \subseteq \ldots \subseteq Z_1$. By Claim 2, we can assume inductively that $Z_{i-1}$ is a union of subsets which are simultaneously $\sigma_{b_j}$-classes for $1 \leq j \leq i$. In particular, $Z_{i-1}$ is a union of sets which are simultaneously $\sigma_{b_{i-1}}$ and $\sigma_{b_i}$-classes. As $g_i \in \Aut(D_\Gamma/b_{i-1})$ and $g_ib_i = b_{i+1}$ it follows that $g_iZ_{i-1}$ is a  union of sets which are simultaneously $\sigma_{b_{i-1}}$ and $\sigma_{b_{i+1}}$-classes. Thus $Z_i = Z_{i-1}\cap g_iZ_{i-1}$ is a union of $\sigma_{b_{i-1}}$-classes, and all of these classes are also $\sigma_{b_j}$-classes for $1\leq j \leq i$ and $j = i+1$, as required. \hfill ($\Box_{Claim\, 3}$)

\medskip

\textit{Claim 4:\/} For $i \geq 1$ we have that $Y_i = \Gamma(b_i)\cap \Gamma(b_{i+1})$ is generated by its intersection with $\Gamma^{\leq n}(b_i)$.

For $i = 1$ this is by choice of $n$. So suppose $i \geq 2$. Then we have $Y_i = Y_{i-1}\cap g_iY_{i-1}$ and $Y_{i-1}, g_iY_{i-1}\subseteq \Gamma(b_{i-1})$. Suppose $y \in Y_i$ is at level $m > n$ in $\Gamma(b_{i-1})$ (and therefore also in $\Gamma^m(b_i)$ and $\Gamma^m(b_{i+1})$, by Claim 1). Then there exist $x \in Y_{i-1}$ and $x' \in g_iY_{i-1}$ which are at level $n$ (in $\Gamma(b_{i-1})$) and which are ancestors of $y$. Then, because of $y$, we have that $\sigma_{b_{i-1}}(x,x')$ holds. But $x \in Z_{i-1}$ and $Z_{i-1}$ is a union of $\sigma_{b_{i-1}}$-classes, so $x' \in Z_{i-1}$ and therefore $x' \in Y_i$. So $y$ has an ancestor in $Z_i$, as required. \hfill ($\Box_{Claim\, 4}$)

\medskip

Suppose now that we can choose the $g_i$ so that if $Z_i \neq \emptyset$ then $g_iZ_i \neq Z_i$, and therefore $Z_{i+1}$ is a proper subset of $Z_i$. As these sets are finite, for some $r$ we must have $Z_r = \emptyset$. It then follows from Claim 4 that $Y_r = \emptyset$, as required by the Proposition.

It remains to explain how to construct the $g_i$, for $i \geq 2$. We do this inductively, so suppose we have constructed up to $g_{i-1}$ and have $b_1,\ldots, b_i$ with the required properties. Suppose $Z_{i-1} \neq \emptyset$. We need to find $g_{i} \in \Aut(D_\Gamma/b_{i-1})$ such that (with the above notation)  $Y_i  \subseteq \Gamma(b_{i-1})$ and $g_iZ_{i-1} \neq Z_{i-1}$. First, note that $Z_{i-1}$ is a proper subset of $\Gamma^n(b_{i-1})$. If not, then $\Gamma^n(b_{i-1}) \subseteq \Gamma^n(b_i)$ and as these are finite sets of the same size, we obtain that $\Gamma^n(b_{i-1}) = \Gamma^n(b_i)$. So $b_{i-1}$ and $b_i$ have the same descendants in level $n$ and as $b_{i-1}\neq b_i$, this contradicts property (1) in the definition of $D_\Gamma$ in Theorem \ref{11}.  It then follows by property P3 of $\Gamma$ that there is an automorphism $h$ of $\Gamma(b_{i-1})$ with $hZ_{i-1} \neq Z_{i-1}$. We claim that $h$ extends to an automorphism $g_i$ of $D_\Gamma$ with the property that $\Gamma(g_ib_i) \cap \Gamma(b_i) \subseteq \Gamma(b_{i-1})$, and this will suffice. 

To do this, let $B \leq^+ D_\Gamma$ be finitely generated and contain $b_{i-1}, b_i$ (by Property (2) of $D_\Gamma$ in Theorem \ref{11}). Applying Lemma \ref{12new} (with the given $h$ and $A = \Gamma(b_{i-1})$) gives the required $g_i$. This completes the proof of Proposition \ref{mainprop2}. \hfill $\Box$

\end{document}